\documentclass{amsart}
\usepackage{amsmath}
\usepackage{amssymb}

\usepackage[all]{xy}

\newcommand{\QQ}{\mathbb{Q}}
\newcommand{\RR}{\mathbb{R}}
\newcommand{\OO}{\mathcal{O}}

\newcommand{\Mov}{\overline{\mov}}
\newcommand{\XX}{\mathcal{X}}
\newcommand{\YY}{\mathcal{Y}}
\newcommand{\ZZ}{\mathcal{Z}}
\newcommand{\PP}{\mathcal{P}}
\newcommand{\WW}{\mathcal{W}}
\newcommand{\CC}{\mathcal{C}}
\newcommand{\volbdff}{\vol_{\operatorname{BdFF}}}
\newcommand{\volmov}{\vol_{\operatorname{mov}}}
\DeclareMathOperator{\Div}{Div} \DeclareMathOperator{\mult}{mult}
\DeclareMathOperator{\ord}{ord} \DeclareMathOperator{\vol}{vol}
\DeclareMathOperator{\mov}{Mov} \DeclareMathOperator{\Bs}{Bs}
\DeclareMathOperator{\SBs}{\textbf{B}}
\DeclareMathOperator{\codim}{codim} \DeclareMathOperator{\im}{im}
\DeclareMathOperator{\Env}{Env} 
\DeclareMathOperator{\Proj}{\textbf{Proj}}

\newtheorem{thm}{Theorem}[section]

\newtheorem{lem}[thm]{Lemma}
\newtheorem{prop}[thm]{Proposition}

\newtheorem{que}[thm]{Question}

\theoremstyle{definition}
\newtheorem{defn}[thm]{Definition}
\newtheorem{ex}[thm]{Example}
\newtheorem{notn}[thm]{Notation}

\theoremstyle{remark}
\newtheorem{rmk}[thm]{Remark}

\numberwithin{equation}{section}

\begin{document}

\title{On isolated singularities with a noninvertible finite endomorphism}

\author{Yuchen Zhang}
\address{Department of Mathematics, University of Michigan, Ann Arbor, MI 48109, USA}
\email{zyuchen@umich.edu}



\date{\today}



\begin{abstract}
We prove that if $\phi:(X,0)\to(X,0)$ is a finite endomorphism of an
isolated singularity such that $\deg(\phi)\geq 2$ and $\phi$ is
\'etale in codimension 1, then $X$ is $\QQ$-Gorenstein and log
canonical.
\end{abstract}

\maketitle

\section{Introduction}

Let us start with an easy example. Let $C$ be a smooth curve of
genus $g$. By an argument using Riemann-Hurwitz Theorem, one can see
that $C$ has a finite endomorphism $\phi$ of degree $\geq 2$ if and
only if $g\leq 1$. In this case, assume that there is an ample
divisor $H$ on $C$ such that $\phi^*H$ is a multiple of $H$. Then
$\phi$ induces a finite endomorphism on the cone $X$ over $C$ with
polarization $mH$, where $m$ is a sufficiently large integer. On the
other hand, one can see easily by adjunction that a normal cone
over a smooth curve of genus $g$  has log canonical singularity if
and only if $g\leq 1$. This phenomenon is true in general. In this
paper, we prove the following theorem:

\begin{thm}\label{main_intro}
Let $(X,0)$ be a normal projective variety with isolated singularity
$0\in X$. Suppose that there exists a finite endomorphism
$\phi:(X,0)\to (X,0)$ such that $\deg\phi\geq 2$ and $\phi$ is
\'etale in codimension 1. Then $X$ is $\QQ$-Gorenstein and log
canonical.
\end{thm}

The assumption that $X$ has isolated singularity is necessary. Otherwise, let $X=E\times V$, where $E$ is an elliptic curve and $V$ is an arbitrary variety with a bad singularity. Then $X$ has an induced noninvertible \'etale endomorphism from $E$.


We briefly review the history of this problem. For the definitions of related terminologies, we refer to Section 2 and \cite{BdFF}.
The surface case is studied in \cite{Wah}. Let $X$ be a normal surface and $f:Y\to X$ be the minimal resolution. The relative Zariski decomposition yields $K_{Y/X}=P+N$. Wahl's invariant is defined as the nonnegative intersection number $-P^2$, which is the key ingredient in the study of surfaces with noninvertible finite endomorphisms. A classification of such surfaces
is given in \cite{Fav,FN}. Wahl's invariant is generalized to higher
dimensions by Boucksom, de Fernex and Favre \cite{BdFF}. Due to the
absence of minimal resolutions, they consider log discrepancy divisors
on all birational models over $X$ as Shokurov's $b$-divisor $A_{\XX/X}$. The
Zariski decomposition is replaced by the nef envelope
$\Env_\XX(A_{\XX/X})$. It can be shown that
$-(\Env_\XX(A_{\XX/X}))^n$ is a well-defined finite nonnegative
number, which is called $\volbdff(X)$. This volume behaves well under
finite morphisms. In particular, they prove the following theorem:

\begin{thm}\cite[Theorem A and B]{BdFF}\cite[Proposition 2.12]{Ful13}
For normal isolated singularities $(X,0)$ with noninvertible finite endomorphism, $\volbdff(X)=0$. Moreover, when $X$ is $\QQ$-Gorenstein, $\volbdff(X)=0$ if and only if
$X$ has log canonical singularity.
\end{thm}

The same theorem is obtained in \cite{BH} by analyzing the behavior of non-log-canonical centers under finite pullback. In \cite{Ful13}, Fulger defines a courser volume $\vol_{\mathrm F}(X)$ as the asymptotic order of growth of plurigenera, which coincides with $\volbdff(X)$ when $X$ is $\QQ$-Gorenstein.

Unfortunately, in \cite{Z14}, the author produces a
non-$\QQ$-Gorenstein isolated singularity $(X,0)$ such that
$\volbdff(X)=0$ while there is no boundary $\Delta$ such that
$(X,\Delta)$ is log canonical. We should remark that, in this
example, $X$ admits a small log canonical modification
\cite{OX}\cite[Proposition 2.4]{BH}.

The $\QQ$-Gorenstein case is further studied in \cite[Section 3]{Z14}.
Specifically, the author shows that, like the surface case,
$\volbdff(X)$ can be calculated by an intersection number on a
certain birational model $f:Y\to X$, namely, the log canonical modification
\cite{OX}. A key property of such a model is that $K_Y+E_f$ is $f$-ample, where
$E_f$ is the reduced exceptional divisor. In the
non-$\QQ$-Gorenstein case, the existence of the log canonical modification is
conjectured to be true assuming the full minimal model program
including the abundance conjecture, but has not yet been proved.

In this paper, in order to show that a normal isolated singularity
$(X,0)$ in Theorem \ref{main_intro} is indeed $\QQ$-Gorenstein, we
consider a birational model over $X$ called a movable modification
(Theorem \ref{mov_modi}) where $K_Y+E_f$ is $f$-movable. The
existence of movable modifications is known to experts. However, we
include a proof in Section 3. We introduce a new volume $\volmov(X)$
(Definition \ref{defn_volmov}) using Nakayama's
$\sigma$-decomposition. We show the following theorem:

\begin{thm}[Proposition \ref{num_lc}, Lemma \ref{vol_pos}, Theorem \ref{main_equiv}, \ref{main_finite} and \ref{main_volmov=0}]\label{main}
For normal isolated singularities $(X,0)$ and $(Y,0)$,
\begin{enumerate}
\item $\volmov(X)$ is a finite nonnegative number, and $\volmov(X)\geq
\volbdff(X)$ with equality when $X$ is $\QQ$-Gorenstein.
\item If $\phi:(Y,0)\to(X,0)$ is a finite morphism of degree $d$ that is \'etale in codimension 1,
then $$\volmov(Y)= d\volmov(X).$$
\item If $\phi:(X,0)\to(X,0)$ is a finite endomorphism of degree $\geq
2$ that is \'etale in codimension 1, then $$\volmov(X)=0.$$
\item If $\volmov(X)=0$, then $X$ is numerically $\QQ$-Gorenstein.
\item Every numerically log canonical variety is $\QQ$-Gorenstein.
\end{enumerate}
\end{thm}

As a byproduct while studying numerically $\QQ$-Gorenstein
varieties, we obtain the following theorem which slightly
generalizes \cite[Theorem 1.1]{OX}.

\begin{thm}[Theorem \ref{lc_modi}]
Let $X$ be a numerically $\QQ$-Gorenstein projective variety. Then
there exists a log canonical model $Y$ over $X$.
\end{thm}



The case that $\phi$ is not \'etale in codimension one is also
interesting. Although $X$ is not $\QQ$-Gorenstein in general (see
example below), we still expect that there exists a boundary
$\Delta$ such that $(X,\Delta)$ is log canonical. It is known that
$X$ has klt singularities assuming that $X$ is $\QQ$-Gorenstein
\cite[Theorem B]{BdFF}.

\begin{ex}
Let $V=\mathbb{P}^1\times E$ where $E$ is an elliptic curve. Let
$\phi_1:\mathbb{P}^1\to \mathbb{P}^1$ be raising to the fourth power
and $\phi_2=[2]:E\to E$ be multiplying by 2. Let $H_1$ be a point in
$\mathbb{P}^1$, $H_2$ be a symmetric ample divisor on $E$ (i.e.,
$[-1]^*H_2=H_2$) and $H=p_1^*H_1+p_2^*H_2$, where $p_i$ are the
projections. Then $\phi=\phi_1\times \phi_2$ is an endomorphism on
$V$ of degree 16 such that $\phi^*H\sim 4H$. We may take $X$ be the
cone over $V$ with polarization $H$. Thus, $\phi$ induces an
endomorphism on $X$ of degree 16. But $X$ is not $\QQ$-Gorenstein or
log terminal. However, one may take $\Delta$ be the cone over
$p_1^*(D_1+D_2)$ for two different points $D_1$ and $D_2$ on
$\mathbb{P}^1$ and see that $K_X+\Delta$ is $\QQ$-Cartier and log
canonical.
\end{ex}

\begin{que} \label{que} Let $(X,0)$ be a normal isolated singularity. If $\phi:(X,0)\to (X,0)$ is a noninvertible
finite endomorphism that is not \'etale in codimension one, does
there exist a boundary $\Delta$ such that $(X,\Delta)$ is log
canonical?
\end{que}

The global counterpart of this problem is well studied. Let $(V,H)$
be a normal polarized projective variety. A finite endomorphism
$\phi:V\to V$ is called polarized if $\phi^*H$ is a multiple of $H$.
The cone over a smooth variety $V$ with the polarization $H$ gives
an isolated singularity as in the local case. The classification of
polarized K\"ahler surfaces (also known as dynamic surfaces) is
obtained in \cite{FN} and \cite[Proposition 2.3.1]{Z06}. The three
dimensional case is studied in \cite{Fuj02} and \cite{FN07}, where
the classification of smooth projective 3-folds with $\kappa(X)\geq
0$ that admit nontrivial endomorphism (which necessarily is
\'etale) is given. Higher dimensions are studied in \cite{NZ09,NZ}\cite[Theorem 1.21]{GKP}.
It is known that a $\QQ$-Gorenstein polarized projective variety
with noninvertible endomorphism is log canonical \cite{BH}. We propose the global version of Question \ref{que}:

\begin{que}
If $\phi$ is a noninvertible polarized finite endomorphism on $(V,H)$, is $V$ log Calabi-Yau? It is known that $-K_V$ is pseudo-effective when $V$ is smooth \cite[Theorem C]{BdFF}.
\end{que}

There are also conditions that are weaker than being polarized
that rise from dynamic systems, such as amplified and unity-free. We
refer to \cite{KR} for the definitions and comparisons.

\subsection*{Acknowledgement} The main work was done during Algebraic Geometry Summer Research Institute organized by AMS and University of Utah. The author would like to thank to Tommaso de Fernex, Christopher Hacon, Lance Miller, Mircea Musta\c{t}\u{a} and Mihnea Popa for inspiring discussions, and the anonymous referee for pointing out several mistakes. The author is partially supported by NSF FRG Grant DMS-1265285.
\section{Preliminaries}

Throughout this paper, we work over an algebraically closed field
$k$ of characteristic 0.

\subsection{Movable divisors}

Let $f:Y\to X$ be a projective morphism between normal varieties and
$D$ be an $f$-big Cartier divisors on $Y$. The base locus of $D$ over $X$,
$\Bs(D)$, is the co-support of the image of the following canonical map:
$$f^*f_*\OO_Y(D)\otimes \OO_Y(-D)\to \OO_Y.$$ The stable base locus
is defined to be $$\SBs(D)=\bigcap_{m\geq 1}\Bs(mD)_{\text{red}}.$$
One can easily extend this definition to $\QQ$-Cartier divisors. We
define the diminished base locus\footnote{This is called restricted
base locus in \cite{ELMNP} and non-nef locus in \cite{BDPP}} of $D$
as
$$\SBs_-(D)=\bigcup_{\epsilon>0} \SBs(D+\epsilon H),$$ where $H$ is
an $f$-ample divisor on $Y$. It can be shown that $\SBs_-(D)$ is
independent of the choice of $H$ (see \cite[Section 1]{ELMNP}.

We say an $f$-pseudo-effective $\QQ$-Cartier divisor $D$ on $Y$ is $f$-mobile if $$\codim(\Bs(D))\geq
2.$$ The \textbf{$f$-movable cone} $\Mov(Y/X)$ is the closure of the
cone generated by $f$-mobile Cartier divisors in the finite dimensional space $N^1(Y/X)$. We call
an $\RR$-divisor $D$ $f$-movable if $D\in\Mov(Y/X)$.

\begin{lem}\label{mov_vs_bs}\cite[Theorem V.1.3]{Nak04}
Let $D$ be a $\QQ$-Cartier divisor on $Y$. Then $D\in\Mov(Y/X)$ if
and only if $\SBs_-(D)$ contains no divisor. \qed
\end{lem}



\begin{lem}\label{mov_factor}
Let $f:Y\to X$ and $g:Z\to Y$ be two projective morphisms and
$\phi=f\circ g$. If a Cartier divisor $D$ on $Z$ is $\phi$-mobile,
then $D$ is also $g$-mobile. In particular, $\Mov(Z/X)\subseteq
\Mov(Z/Y)$.
\end{lem}
\begin{proof}
Since the morphism $$g^*f^*f_*g_*\OO_Z(D)\otimes
\OO_Z(-D)=\phi^*\phi_*\OO_Z(D)\otimes
\OO_Z(-D)\xrightarrow{\rho_\phi}\OO_Z$$ factors through
$$g^*g_*\OO_Z(D)\otimes \OO_Z(-D)\xrightarrow{\rho_g}\OO_Z,$$ it
follows that $\im \rho_\phi\subseteq \im \rho_g$. Thus, if the
co-support of $\im \rho_\phi$ has codimension $\geq 2$, then so is
the co-support of $\im \rho_g$. The lemma follows.
\end{proof}

The following lemma is a generalization of the well-known Negativity
Lemma \cite[Lemma 3.39]{KM}.
\begin{lem}\label{neg_lemma_mov}\cite[Lemma 4.2]{Fuj11} Let
$f:Y\to X$ be a proper birational morphism where $Y$ is a normal
$\QQ$-factorial variety. Let $E$ be an $\RR$-divisor on $Y$ such
that $E$ is $f$-exceptional and $E\in\Mov(Y/X)$. Then $E\leq 0$.
\qed
\end{lem}

\subsection{Shokurov's $b$-divisor}

Let $X$ be a normal variety.

A \textbf{Weil $b$-divisor} $\WW$ over $X$ is the assignment to each
birational model $f:Y\to X$, a Weil divisor $\WW_Y$ on $Y$ that is compatible with pushforwards:
if $\phi:Z\to Y$ are two models over $X$, then
$\phi_*(\WW_Z)=\WW_Y$. The Weil-divisor $\WW_Y$ is called the
\textbf{trace} of $\WW$ on $Y$. We denote by $\Div(\XX)$ the group
of Weil $b$-divisors over $X$ and define $\QQ$-Weil ($\RR$-Weil,
resp.) $b$-divisor as elements of $\Div(\XX)\otimes \QQ$
($\Div(\XX)\otimes \RR$, resp.).

We call the Weil $b$-divisor $\CC$ a \textbf{Cartier $b$-divisor}
over $X$ if there exists a birational model $f:Y\to X$ such that
$\CC_Y$ is Cartier and for every other model $\phi:Z\dashrightarrow
Y$ with common resolution
$$\xymatrix{&W\ar_{s}[ld]\ar^{t}[rd]&\\Z\ar@{-->}[rr]&&Y}$$ we have
$\CC_Z=s_*(t^*\CC_Y)$. In this case, we say that $f:Y\to X$ is a
\textbf{determinant} of $\CC$. Similarly, we define $\QQ$-Cartier
$b$-divisor and $\RR$-Cartier $b$-divisor.

A $\RR$-Cartier $b$-divisor $\CC$ over $X$ is called
\textbf{$X$-nef} if there exists a (hence any) determinant $f:Y\to
X$ such that $\CC_Y$ is $f$-nef. We call a $\RR$-Weil $b$-divisor
$\WW$ $X$-nef if $\WW$ is a limit of $X$-nef $\RR$-Cartier
$b$-divisors, where the limit is taken in the numerical class of
every model. The following lemma gives a characterization of $X$-nef
$\RR$-Weil $b$-divisors.

\begin{lem}\label{nef_vs_mov}\cite[Lemma 2.10]{BdFF}
A $\RR$-Weil $b$-divisor $\WW$ is $X$-nef if and only if
$\WW_Y\in\Mov(Y/X)$ on every $\QQ$-factorial model $Y$. \qed
\end{lem}

Let $\WW_1$ and $\WW_2$ be two $\RR$-Weil $b$-divisors. We say
$\WW_1\geq \WW_2$, if for every model $Y$, we have $(\WW_1)_Y\geq
(\WW_2)_Y$.

\subsection{Nakayama's $\sigma$-decomposition}


Let $X$ be a $\QQ$-factorial projective normal variety over $S$ and $D$ be an $S$-big $\RR$-divisor on $X$. For every prime divisor $\Gamma$ on $X$, we define
$$\sigma_\Gamma(D)=\inf\{\mult_\Gamma\Delta | \Delta\equiv_S D, \Delta\geq0\}.$$
Since $D$ is $S$-big, there exists an effective divisor $\Delta$
that is numerically equivalent to $D$. Hence, the infimum above is
not taken over an empty set.

\begin{rmk}
By the observation of Nakayama, this definition can be generalized
to $S$-pseudo-effective $\RR$-divisors. Let $D$ be an
$S$-pseudo-effective divisor on $X$. Fix an $S$-ample divisor $A$ on
$X$. Since for every $\epsilon>0$, $D+\epsilon A$ is $S$-big, we can
define $\sigma_\Gamma (D)=\lim_{\epsilon\downarrow
0}\sigma_\Gamma(D+\epsilon A)$. It is shown in \cite[Lemma
III.1.4-5]{Nak04} that $\sigma_\Gamma(D)$ is independent of the
choice of $A$ and only depends on the numerical class of $D$. However,
an example is given in \cite{Les15} that this limit can be $\infty$
when $S$ is not a point. In this paper, we only use the
$\sigma$-decomposition in the case that $X$ is birational to $S$,
hence every divisor is $S$-big (see \cite[Lemma
III.1.4(2)]{Nak04}).
\end{rmk}

It is shown in \cite[Lemma III.4.2]{Nak04} that there are finitely
many prime divisors $\Gamma$ on $X$ such that $\sigma_\Gamma(D)>0$,
which leads us to the following definition.

\begin{defn}
Let $D$ be an $S$-big $\RR$-divisor on a $\QQ$-factorial projective
variety $X$ over $S$. We define
$$N_\sigma(D)=\sum\sigma_\Gamma(D)\Gamma \text{\quad and \quad}
P_\sigma(D)=D-N_\sigma(D),$$ where $P_\sigma(D)$ and $N_\sigma(D)$
are called the positive and negative part of the
$\sigma$-decomposition, respectively.
\end{defn}

\begin{notn}
To be precise, the $\sigma$-decomposition here should be written as
$P_{\sigma/S}$ and $N_{\sigma/S}$ as we are using the relative
version. However, in order to make our notation concise, we will
omit the base $S$ when it is clear from the context.
\end{notn}

We record some properties of the $\sigma$-decomposition as below.

\begin{prop}\label{sigma_decom}
Let $D$ be an $S$-big $\RR$-divisor on a $\QQ$-factorial variety
$X$. Let $f:Y\to X$ be a proper birational morphism, where $Y$ is
normal.
\begin{enumerate}
\item $N_\sigma(D)=0$ if and only if $D\in\Mov(X/S)$.
\item $P_\sigma(D)$ is the largest $\RR$-divisor in $\Mov(X/S)$ that is no greater than $D$.
\item $f_*P_\sigma(f^*D)=P_\sigma(D)$. In other words, $P_\sigma(f^*D)$ defines a $\RR$-Weil $b$-divisor over $X$.
\item $\sigma_\Gamma$ is a continuous function on the big cone of $X$ over $S$. So is $P_\sigma$, where the convergence is coefficiently-wise.
\end{enumerate}
\end{prop}

\begin{proof}
(1) and (2) are \cite[Proposition III.1.14]{Nak04}. (3) is
\cite[Theorem III.2.5(1)]{Nak04}. (4) is \cite[Lemma
III.1.7(1)]{Nak04}.
\end{proof}

Inspired by the above proposition, we give the definition of
$\sigma$-closure.

\begin{defn}
Let $Y$ be a $\QQ$-factorial model over $X$ and $D$ be an $S$-big
$\RR$-divisor on $Y$. The $\sigma$-closure $\PP_\sigma(D)$ is
an $\RR$-Weil $b$-divisor such that: for every model $Z$ over $X$, let $W$ be
a model dominating $Y$ and $Z$.
$$\xymatrix{
& W \ar_{s}[ld] \ar^{t}[rd] & \\
Y \ar@{-->}[rr]&& Z }$$ Then the trace of $\PP_\sigma(D)$ on $Z$ is
$t_*P_\sigma(s^*D)$. It is well-defined by Proposition
\ref{sigma_decom}(3).
\end{defn}

\begin{lem}\label{sigma_is_nef}
The $\sigma$-closure is $X$-nef.
\end{lem}
\begin{proof}
The lemma follows directly from Proposition \ref{sigma_decom}(2) and
Lemma \ref{nef_vs_mov}.
\end{proof}

\subsection{Pullback of Weil divisors}

We recall the definition of pullback of Weil divisors from
\cite{dFH,BdFF} as below.

Let $X$ be a normal variety, $D$ be a $\RR$-Weil divisor on $X$ and
$E$ be a prime divisor over $X$. We define the valuation
$v_E(D)=v_E(\OO_X(-D))=v_E(\OO_X(\lfloor -D\rfloor))$ as
$$v_E(D)=\min\{v_E(\phi)\ |\ \phi\in\OO_U(\lfloor -D\rfloor), U\cap
c_X(E)\neq \emptyset\}.$$ If $D$ is Cartier, we see that $v_E(D)$ is
the usual valuation.

Suppose $f:Y\to X$ is a proper birational morphism and $Y$ is
normal. The natural pullback $f^\natural D$ is defined as
$$f^\natural D=\sum v_E(D)E,$$ where the sum runs through all prime
divisors on $Y$. It is easy to see that $\OO_Y(-f^\natural
D)=(\OO_X(-D)\cdot\OO_Y)^{\vee\vee}$. The natural pullback is also
known as $Z(\OO_X(-D))_f$ in \cite{BdFF}, where $Z(\OO_X(-D))$ is
viewed as a Weil $b$-divisor consisting of natural pullbacks.

In general, the natural pullbacks do not behave well under
composition.

\begin{lem}\label{composition_natural}\cite[Lemma 2.7]{dFH}
Let $f:Y\to X$ and $g:Z\to Y$ be two proper birational morphisms
between normal varieties, D be an $\RR$-Weil divisor on
$X$. Then $(f\circ g)^\natural D-g^\natural(f^\natural D)$ is
not necessarily zero . However, it is effective and $g$-exceptional.
Moreover, if $\OO_X(-D)\cdot\OO_Y$ is an invertible sheaf on $Y$, then
$(f\circ g)^\natural D=g^\natural(f^\natural D)$. \qed
\end{lem}

\begin{lem}\label{anti_free_natural}\cite[Lemma 1.8 and
2.10]{BdFF}
\begin{enumerate}
\item \label{pt_gg} If $\OO_X(-D)\cdot\OO_Y$ is an invertible sheaf, then
$-f^\natural D$ is relatively globally generated over $X$. In
other words, the following canonical homomorphism is surjective.
$$f^*f_*\OO_Y(-f^\natural D)\to \OO_Y(-f^\natural D).$$
\item If $-f^\natural D$ is $\QQ$-Cartier, then $-f^\natural D\in
\Mov(Y/X)$.
\end{enumerate}
\end{lem}

\begin{proof}
\begin{enumerate}
\item Since $\OO_X(-D)\cdot\OO_Y$ is the image of $f^*\OO_X(-D)$ under
the homomorphism $f^*\OO_X(-D)\to f^*\OO_X$, we have a surjection
$$f^*\OO_X(-D)\to\OO_X(-D)\cdot\OO_Y.$$ It is obvious that
$f_*\OO_Y(-f^\natural D)=\OO_X(-D)$. The statement follows.
\item Suppose $-mf^\natural D$ is Cartier. Let $g:Z\to X$ be a log
resolution of $X$ such that $\OO_X(-D)\cdot \OO_Z$ is an invertible
sheaf and $g$ factors through $f$ via $\phi:Z\to Y$. Such log
resolution exists by \cite[Theorem 4.2]{dFH}. By part (\ref{pt_gg}),
$-mg^\natural D$ is relatively globally generated over $X$. Hence,
$-mf^\natural D=\phi_*(-mg^\natural D)$ is $f$-mobile.
\end{enumerate}
\end{proof}

It is shown in \cite[Lemma 2.8]{dFH} that for every positive integer
$m$, we have $f^\natural D\geq \frac{1}{m}f^\natural(mD)$. Thus, we
can define the pullback of $D$ as a $\RR$-Weil divisor,
$$f^*D=\liminf_{m\to\infty}\frac{f^\natural(mD)}{m}.$$ It is not
hard to see that the $\liminf$ above is actually a limit
(\cite[Lemma 2.1]{BdFF}).

\begin{lem}\label{anti_movable}
If $Y$ is $\QQ$-factorial, then $-f^*D\in\Mov(Y/X)$.
\end{lem}
\begin{proof}
This is obvious by Lemma \ref{anti_free_natural} and the definition.
\end{proof}

\begin{lem}\label{dp}
Let $f:Y\to X$ and $g:Z\to Y$ be two proper birational morphisms
between normal varieties. Then for every $\QQ$-Weil divisor $D$ on
$X$, $(f\circ g)^*D-g^*(f^*D)$ is an effective $g$-exceptional
divisor.
\end{lem}
\begin{proof}
The lemma follows from Lemma \ref{composition_natural} and
definition.
\end{proof}

We will use the following notation from \cite[Remark 2.4]{BdFF}.
\begin{defn} $\Env_X(D)$ is the $\RR$-Weil $b$-divisor whose trace
on every model $Y$ is $-f^*(-D)$.
\end{defn}
It is shown in \cite[Corollary 2.13]{BdFF} that $\Env_X(D)$ is the
largest $X$-nef $\RR$-Weil $b$-divisor $\WW$ such that $\WW_X=D$.

\begin{lem}\label{double_pullback}
In the setting of Lemma \ref{dp}, if $Z$ is $\QQ$-factorial, then
$-(f\circ g)^*D=P_\sigma(-g^*(f^*D))$. Hence,
$\Env_X(-D)=\PP_\sigma(-f^*D)$.
\end{lem}

\begin{proof}
By Lemma \ref{anti_movable}, $-(f\circ g)^*D$ is a movable
$\RR$-divisor that is no greater than $-g^*(f^*D)$. Using Lemma
\ref{sigma_decom}, we have $-(f\circ g)^*D\leq
P_\sigma(-g^*(f^*D))$. On the other hand, $\PP_\sigma(-f^*D)$ is an
$X$-nef $\RR$-Weil $b$-divisor whose trace on $X$ is $D$. Hence,
$\Env_X(-D)\geq \PP_\sigma(-f^*D)$. The lemma follows.
\end{proof}

\begin{lem}\label{nef_determinant}\cite[Corollary 2.13]{BdFF}
If $\Env_X(D)$ is $\RR$-Cartier with a determinant on $Y$, then the trace of $\Env_X(D)$ on $Y$ is $X$-nef. \qed
\end{lem}

We record the following Negativity Lemma in the context of $b$-divisors.

\begin{lem}\label{neg_lemma_b}\cite[Proposition 2.12]{BdFF}
Let $\WW$ be an $X$-nef $\RR$-Weil $b$-divisor over $X$ and $f:Y\to
X$ be a birational model. Then $\WW_Y\leq -f^*(-\WW_X)$. \qed
\end{lem}

\subsection{Nef envelope of $\RR$-Weil $b$-divisors and volume}
All the definitions below are from \cite{BdFF}, where the reader can
find the details.

Let $\WW$ be an $\RR$-Weil $b$-divisor over $X$. We define the nef
envelope $\Env_\XX(\WW)$ of $\WW$ to be the largest $X$-nef
$\RR$-Weil $b$-divisor $\mathcal{Z}$ such that $\mathcal{Z}\leq
\WW$, if exists. Let $\CC$ be an $\RR$-Cartier $b$-divisor over $X$ whose center on $X$ is $0\in X$.
The self-intersection $\CC^n=\CC_Y^n$ if $Y$ is a determinant of
$\CC$. It is obvious that $\CC^n$ is independent of the choice of
$Y$. If $\WW$ is $X$-nef whose center on $X$ is $0\in X$, we define $\WW^n=\inf\{\CC^n\}$ where the
infimum is taken for all $X$-nef $\RR$-Carter $b$-divisors such that
$\CC\geq \WW$.

We copy the following properties from \cite{BdFF}.

\begin{prop}\label{self_inter} \cite[Theorem 4.14 and Proposition 4.16]{BdFF}
Let $\WW$ be an $X$-nef $\RR$-Weil $b$-divisor whose center on $X$ is $0\in X$.
    \begin{enumerate}
    \item If $\WW_1\leq \WW_2$ are two $X$-nef $\RR$-Weil
    $b$-divisors over $X$, then $\WW_1^n\leq \WW_2^n\leq 0$.
    \item $\WW^n=0$ if and only if $\WW=0$.
    \item Let $\phi:(Y,0)\to(X,0)$ be a finite map of degree $d$ between isolated
    singularities. Then
    $(\phi^*\WW)^n=d(\WW^n)$. \qed
    \end{enumerate}
\end{prop}

For a model $f:Y\to X$, we fix canonical divisors such that
$f_*K_Y=K_X$. For every positive integer $m$, we define the
\textbf{$m$-th limiting log discrepancy divisor} as
$$A_{m,Y/X}=K_Y+E_f-\frac{1}{m}f^\natural(mK_X),$$ where $E_f$ is
the reduced exceptional divisor. We denote by $A_{m,\XX/X}$ the
corresponding $\QQ$-Weil $b$-divisor. Similarly, we define the
\textbf{log discrepancy divisor} as $$A_{Y/X}=K_Y+E-f^*K_X$$ and
$A_{\XX/X}$ to be the corresponding $\RR$-Weil $b$-divisor.

It is shown in \cite{BdFF,Z14} that if $X$ has isolated singularity,
then $\Env_\XX(A_{\XX/X})$ and $\Env_\XX(A_{m,\XX/X})$ exist. We
define
$$\vol_m(X)=-\Env_\XX(A_{m,\XX/X})^n \text{\quad and \quad}
\volbdff(X)=-\Env_\XX(A_{\XX/X})^n.$$ It is known that these volumes
are finite nonnegative numbers.

In the case that $X$ is $\QQ$-Gorenstein, it is shown in
\cite[Theorem 3.2]{Z14} that $\volbdff(X)=\vol_m(X)=-A_{Y/X}^n$ for $m$ sufficiently large and divisible,
where $Y$ is the log canonical modification of $X$. When $X$ is
non-$\QQ$-Gorenstein, \cite[Corollary 4.3]{Z14} shows that for every
$m\geq 2$, one can always pick a boundary $\Delta$ on $X$ such that
$\vol_m(X)$ is calculated on the log canonical modification of the
pair $(X,\Delta)$. We do not know whether $\volbdff(X)$ can be
calculated by intersection number on some model.

The following theorem gives a criteria for log canonical singularity.

\begin{thm}\label{vol_m=0}\cite[Corollary 4.6]{Z14}
If $\vol_m(X)=0$ for some (hence any multiple of an) integer $m\geq 1$,
then there exists a boundary $\Delta$ on $X$ such that $(X,\Delta)$
is log canonical. \qed
\end{thm}

\subsection{Numerically $\QQ$-Cartier divisors}

The numerically $\QQ$-Cartier divisors are defined in \cite{BdFF} and are further studied in \cite{BdFFU}, which behave like $\QQ$-Cartier divisors under birational pullbacks. We give the following definition which generalizes Mumford's numerical pullback.

\begin{defn} Let $X$ be a normal variety.
\begin{enumerate}
\item A Weil divisor $D$ on $X$ is numerically Cartier if there exists
a resolution of singularities $f:Y\to X$ and an $f$-numerically
trivial Cartier divisor $D'$ on $Y$ such that $f_*D'=D$.
\item A $\QQ$-Weil divisor is numerically $\QQ$-Cartier if some
multiple is numerically Cartier.
\item When the canonical divisor $K_X$ is numerically $\QQ$-Cartier, we say $X$ is numerically $\QQ$-Gorenstein.
\end{enumerate}
\end{defn}

\begin{prop}\label{num_Car}
Let $D$ be a Weil divisor on $X$.
\begin{enumerate}
\item $D$ is numerically $\QQ$-Cartier if and only if
$\Env_X(-D)=-\Env_X(D)$.
\item Suppose $X$ has rational singularities. Then $D$ is
numerically $\QQ$-Cartier if and only if $D$ is $\QQ$-Cartier.
\item Suppose $D$ is numerically $\QQ$-Cartier. Let $f:Y\to X$ be a birational model such that $f^*D$ is $\QQ$-Cartier. Then $f^*D$ is $f$-numerically trivial and $-\Env_X(-D)=\overline{f^*D}$.
\end{enumerate}
\end{prop}

\begin{proof}
(1) is \cite[Proposition 5.9]{BdFFU} and (2) is \cite[Theorem 5.11]{BdFFU}.

For (3), let $g:Z\to Y$ be any $\QQ$-factorial birational model over $Y$ and $\phi=f\circ g$. By Lemma \ref{double_pullback}, we have
$$-\phi^*D=P_\sigma(-g^*(f^*D))\leq -g^*(f^*D),$$ and $$-\phi^*(-D)=P_\sigma(-g^*(f^*(-D)))\leq -g^*(f^*(-D)).$$ Applying (1), the sum yields $$0=-\phi^*D-\phi^*(-D)\leq -g^*(f^*D)-g^*(f^*(-D))=0.$$ Hence, $\Env_X(D)=\overline{f^*D}$ and $-\Env_X(D)=\overline{-f^*D}$. By Lemma \ref{nef_determinant}, both $f^*D$ and $-f^*D$ are $f$-nef. Thus, $f^*D$ is $f$-numerically trivial.
\end{proof}

\section{Partial Resolutions}

\subsection{Movable modification of non-$\QQ$-Gorenstein varieties}

The construction of our volume depends on the following theorem
which is known to experts.

\begin{thm}\label{mov_modi}
There exists a birational morphism $f:Y\to X$ such that
\begin{enumerate}
\item $Y$ is $\QQ$-factorial,
\item $(Y,E_f)$ is dlt, and
\item $K_Y+E_f\in\Mov(Y/X)$,
\end{enumerate}
where $E_f$ is the reduced exceptional divisor.
Moreover, $A_{Y/X}\leq 0$. We call $f:Y\to X$ a movable modification of $X$.
\end{thm}

\begin{rmk}
Theorem \ref{mov_modi} does not require that $X$ is
$\QQ$-Gorenstein. It is proved in \cite{OX} that in the setting of
$K_X+\Delta$ being $\QQ$-Cartier, one may further require that
$K_Y+E_f+\tilde\Delta$ is $f$-semi-ample, where $\tilde\Delta$ is
the strict transform of $\Delta$ on $Y$. The general case of
semi-ampleness is true if one assume the full minimal model program
(MMP) including abundance.
\end{rmk}

\begin{proof}
Let $g:Z\to X$ be a log resolution and $E_g$ be the reduced
exceptional divisor. Let $H$ be a $g$-ample divisor on $Z$. We run
the $(K_Z+E_g)$-MMP with scaling $H$ over $X$. We obtain a sequence
of flips and divisorial contractions:
$$Z=Z_0\dashrightarrow Z_1\dashrightarrow \cdots,$$ and a decreasing
sequence $$\lambda_i=\inf\{s\in\RR | K_{Z_i}+E_{g_i}+sH_i \text{\
is nef over\ }X\}.$$ We know that for every $\lambda_i\geq t\geq
\lambda_{i+1}$, $K_{Z_i}+E_{g_i}+tH_i$ is relatively semi-ample over
$X$ and $\lim_{i\to\infty} \lambda_i=0$ (see \cite[Lemma 2.6]{OX}).

For every divisor $E\subset\SBs_-(K_Z+E_g/X)$ or equivalently every
component $E$ of $N_\sigma(K_Z+E_g)$, there is some $t>0$ such that
$E\subset\SBs(K_Z+E_g+tH/X)$. We may find an $i$ such that
$\lambda_i\geq t\geq \lambda_{i+1}$. Since $K_{Z_i}+E_{g_i}+tH_i$ is
relatively semi-ample over $X$, $E$ must be contracted on $Z_i$.
Since there are finitely many such $E$, we conclude that there is a
model $Z_j$ such that $\SBs_-(K_{Z_j}+E_{g_j}/X)$ contains no
divisor.

By Lemma \ref{mov_vs_bs}, we have $K_{Z_j}+E_{g_j}\in\Mov(Z_j/X)$.
The theorem follows by setting $Y=Z_j$.

For the moreover part, by Lemma \ref{anti_movable},
$-f^*K_X\in\Mov(Y/X)$, hence so is $K_Y+E_f-f^*K_X$. By Lemma
\ref{neg_lemma_mov}, we have $A_{Y/X}=K_Y+E_f-f^*K_X\leq 0$.
\end{proof}

\subsection{Log canonical modification of numerically $\QQ$-Gorenstein varieties}

\begin{defn} For a numerically $\QQ$-Gorenstein variety $X$, we say $X$ is numerically log canonical (numerically log terminal, resp.) if for every exceptional divisor $E$ on $\phi:Z\to X$, $\ord_E(K_Z-f^*K_X)\geq -1$ ($>-1$, resp.).
\end{defn}

\begin{rmk}
The above definition coincide with Mumford's numerically pullbacks \cite[Notation 4.1]{KM} on surfaces.
\end{rmk}

The following proposition generalizes the well-known property in the surface case \cite[Notation 4.1]{KM}\cite[Proposition 7.14]{dFH}.

\begin{prop}\label{num_lc}
If $X$ is numerically log canonical, then $X$ is $\QQ$-Gorenstein.
\end{prop}

\begin{proof}
Let $f:Y\to X$ be a movable modification of $X$. Then $K_Y+E_f-f^*K_X\leq 0$ by Lemma \ref{neg_lemma_mov}. Since $X$ is numerically log canonical, there must be $K_Y+E_f=f^*K_X$. Thus, $(Y,E_f)$ is a dlt pair such that $K_Y+E_f$ is $f$-numerically trivial. By the abundance theorem \cite[Theorem 4.9]{FG}, $K_Y+E_f$ is $f$-semi-ample. Let $\phi:Z=\Proj_X\bigoplus_{m}f_*\OO_Y(m(K_Y+E_f))\to X$ be the log canonical modification over $X$. Then $K_Z+E_\phi$ is $\phi$-ample. On the other hand, $K_Z+E_\phi=\phi^*K_X$ is $\phi$-numerically trivial by Proposition \ref{num_Car}(3). We conclude that $\phi$ is an isomorphism. In particular, $X$ is $\QQ$-Gorenstein.
\end{proof}

The existence of log canonical modifications is predicted by full relative minimal model program and is proved for a pair $(X,\Delta=\sum a_i\Delta_i)$ such that $K_X+\Delta$ is $\QQ$-Cartier and $a_i\in[0,1]$ \cite{OX}. The same idea applies in the case that $K_X+\Delta$ is numerically $\QQ$-Cartier. For the reader's convenience, we include a sketch below for the case $\Delta=0$.

\begin{thm}\label{lc_modi} Let $X$ be a numerically $\QQ$-Gorenstein normal variety. Then there exists a unique birational model $f:Y\to X$, such that
\begin{enumerate}
\item $K_Y+E_f$ is a $f$-ample $\QQ$-Cartier divisor, and
\item $(Y,E_f)$ is log canonical,
\end{enumerate}
where $E_f$ is the reduced exceptional divisor.
\end{thm}

\begin{proof}
\emph{Step 1.} By Theorem \ref{mov_modi}, there exists a movable modification $\phi:Z\to X$ such that $Z$ is $\QQ$-factorial, $K_Z+E_\phi\in\Mov(Z/X)$ and $(Z,E_\phi)$ is dlt. By Lemma \ref{mov_vs_bs}, we know that $\SBs_-(K_Z+E_\phi/X)$ contains no divisor. We denote the complement of the support of $\phi(-\lfloor A_{Z/X}\rfloor)$ by $X_{lc}$.

\emph{Step 2.} We show that if $E$ is an exceptional divisor on $Z$ such that $\phi(E)\subseteq X\backslash X_{lc}$, then $\ord_E(A_{Z/X})<0$. If this is not true, by Koll\'ar-Shokurov's Connectedness Lemma \cite[Corollary 5.49]{KM}, there exist exceptional divisors $E_0$ and $E_1$ with centers in $X\backslash X_{lc}$ such that $\ord_{E_0}(A_{Z/X})=0$, $\ord_{E_1}(A_{Z/X})<0$ and $E_0\cap E_1\neq \emptyset$. Then $(A_{Z/X}+\epsilon H)|_{E_0}$ is not effective for $0<\epsilon \ll 1$ and $H$ ample. Since $\phi^*X$ is numerically trivial by Proposition \ref{num_Car}(3), we have $E_0\subseteq\SBs_-(K_Z+E_\phi/X)$, a contradiction.

\emph{Step 3.} Set $B=E_\phi-\epsilon A_{Z/X}$ for some small $\epsilon$ such that $B\geq 0$. We show that $(Z,B)$ has a good minimal model over $X$. The idea is to apply \cite[Theorem 1.1]{HX} over the open subset $X_{lc}$. A good minimal model exists over $X_{lc}$ by Proposition \ref{num_lc} and \cite[Lemma 2.11]{HX}. One can check that no strata of $\lfloor B\rfloor$ are contained in $\phi^{-1}(X\backslash X_{lc})$ using Step 2.

\emph{Step 4.} Since $K_Z+B= (1+\epsilon)(K_Z+E_\phi)-\epsilon\phi^*K_X\equiv_f (1+\epsilon)(K_Z+E_\phi),$ a good minimal model of $K_Z+B$ is also a good minimal model of $K_Z+E_\phi$. Hence the canonical ring $R=\bigoplus_m \phi_*\OO_Z(m(K_Z+E_\phi))$ is finitely generated over $X$ and we may take $Y=\Proj_X R$.
\end{proof}

\section{Movable Volume}

We start with a lemma showing the behavior of Zariski decomposition
in the sense of \cite{BdFF} on a movable modification.

\begin{lem}\label{trace_on_mov}
Let $X$ be a normal variety which is not necessarily
$\QQ$-Gorenstein. Let $f:Y\to X$ be a movable modification. Then the
trace of $\Env_\XX(A_{\XX/X})$ on $Y$ is $A_{Y/X}$.
\end{lem}
\begin{proof}
Consider the $\RR$-Weil $b$-divisor
$\PP_{-}=\PP_\sigma(K_Y+E_f)+\Env_X(-K_X)$. Let $g:Z\to Y$ be a
projective birational morphism and $\phi=g\circ f$. Since $(Y,E_f)$
is dlt, we have $g^*(K_Y+E_f)\leq K_Z+E_\phi$. Hence,
$$(\PP_-)_Z = \PP_\sigma(g^*(K_Y+E_f))-\phi^*K_X \leq g^*(K_Y+E_f)-\phi^*K_X
\leq  A_{Z/X}.$$ Thus, $\PP_-$ is an $X$-nef $\RR$-Weil $b$-divisor
that is less than or equal to $A_{\XX/X}$. By the definition of nef
envelope, we have $\PP_-\leq \Env_\XX(A_{\XX/X})$. Taking the trace
on $Y$, we obtain $$A_{Y/X}=(\PP_-)_Y\leq
(\Env_\XX(A_{\XX/X}))_Y\leq (A_{\XX/X})_Y= A_{Y/X}.$$ The lemma
follows.
\end{proof}

Inspired by the lemma above, we give the following definition:

\begin{defn}
Let $X$ be a normal variety and $f:Y\to X$ be a movable modification. The diminished positive part $\PP_-$ is an $\RR$-Weil $b$-divisor over $X$: $$\PP_-=\PP_\sigma(K_Y+E_f)+\Env_X(-K_X).$$
\end{defn}

\begin{lem}\label{well_def}
Let $f:Y\to X$ be a movable modification, $g:Z\to Y$ be a
$\QQ$-factorial birational model and $\phi=f\circ g$. Then the
trace of $\PP_-$ on $Z$ is $P_\sigma(K_Z+E_\phi)-\phi^*K_X$. In
particular, $\PP_-$ is independent of the choice of $Y$.
\end{lem}
\begin{proof}
We only need to show that
$$P_\sigma(K_Z+E_\phi)=P_\sigma(g^*(K_Y+E_f)).$$ Since $(Y,E_f)$ is
dlt, we have $K_Z+E_\phi\geq g^*(K_Y+E_f).$ Hence,
$$P_\sigma(K_Z+E_\phi)\geq P_\sigma(g^*(K_Y+E_f)).$$ Now we have
$$P_\sigma(g^*(K_Y+E_f))\leq P_\sigma(K_Z+E_\phi)\leq K_Z+E_\phi.$$
As we know that $K_Y+E_f\in\Mov(Y/X)$, by Proposition
\ref{sigma_decom}(3), $g^*(K_Y+E_f)-P_\sigma(g^*(K_Y+E_f))$ is
$g$-exceptional, and so is $K_Z+E_\phi-P_\sigma(g^*(K_Y+E_f))$. We
obtain that $P_\sigma(K_Z+E_\phi)-g^*(K_Y+E_f)$ must be
$g$-exceptional. Notice that
$P_\sigma(K_Z+E_\phi)-g^*(K_Y+E_f)\in\Mov(Z/Y)$. By Lemma
\ref{neg_lemma_mov}, we have $P_\sigma(K_Z+E_\phi)\leq
g^*(K_Y+E_f)$. Proposition \ref{sigma_decom}(2) gives the desired
inequality.
\end{proof}

\begin{rmk}
If we assume the termination of MMP for dlt pairs, then there exists
a movable modification with $K_Y+E_f$ nef over $X$, which is called
a dlt modification \cite{OX}. In this case, it is not hard to see
that $\PP_\sigma(K_Y+E_f)$ is a $\QQ$-Cartier $b$-divisor determined
by $f$.
\end{rmk}

\begin{lem}\label{gen_trace}
Let $\phi:Z\to X$ be a $\QQ$-factorial birational model over $X$ not
necessarily factoring through a movable modification. Then
$(\PP_-)_Z\leq P_\sigma(K_Z+E_\phi)-\phi^*K_X$.
\end{lem}
\begin{proof}
Let $f:Y\to X$ be a movable modification and $W$ be a common
resolution of $Y$ and $Z$ as below:
$$\xymatrix{
& W\ar_{s}[ld]\ar^{t}[rd] & \\
Y & & Z }$$ We only need to show that $t_*P_\sigma(s^*(K_Y+E_f))\leq
P_\sigma(K_Z+E_\phi)$. It is clear that
$$t_*P_\sigma(s^*(K_Y+E_f))\leq t_*(s^*(K_Y+E_f))\leq
t_*(K_W+E_{f\circ s})=K_Z+E_\phi$$ and
$t_*P_\sigma(s^*(K_Y+E_f))\in\Mov(Z/X)$. The lemma now follows from
Proposition \ref{sigma_decom}(2).
\end{proof}

\begin{defn}\label{defn_volmov}
Suppose that $(X,0)$ has only isolated singularity at $0\in X$. We
define the \textbf{movable volume} $\volmov(X)=-(\PP_-)^n$.
\end{defn}

\begin{lem}\label{vol_pos}
We have the following inequality between the volumes:
$$0\leq \volbdff(X)\leq \volmov(X)<+\infty.$$
\end{lem}
\begin{proof}
According to \cite[1.12.1]{Kol13}, there exists a log resolution $s:W\to X$
with an effective $s$-exceptional divisor $H$ on $W$ such that $-H$
is $s$-ample. We fix an integer $m$ such that $K_W+E_s-mH$ and
$-s^*K_X-mH$ are both $s$-ample. For any model $\phi:Z\to X$
dominating a movable model and factoring through $s$ via $t:Z\to W$,
we have
$$\begin{aligned}(\PP_-)_Z & =P_\sigma(K_Z+E_\phi)-\phi^*K_X \\
& \geq P_\sigma(t^*(K_W+E_s))+P_\sigma(-t^*(s^*K_X)) \\
& \geq P_\sigma(t^*(K_W+E_s-mH))+P_\sigma(t^*(-s^*K_X-mH)) \\
& = t^*(K_W+E_s-mH)+t^*(-s^*K_X-mH) \\
& = t^*(A_{W/X}-2mH).
\end{aligned}$$

Thus, by Proposition \ref{self_inter},
$$\volmov(X)=-(\PP_-)^n\leq -(A_{W/X}-2mH)^n<+\infty.$$

By the proof of Lemma \ref{trace_on_mov}, we have $\PP_-\leq
\Env_\XX(A_{\XX/X})$. Hence,
$$\volmov(X)=-(\PP_-)^n\geq -(\Env_\XX(A_{\XX/X}))^n=\volbdff(X)\geq
0.$$
\end{proof}

\begin{thm}\label{main_equiv}
$\volmov(X)=0$ if and only if $X$ is $\QQ$-Gorenstein and log canonical.
\end{thm}


\begin{proof}
Suppose that $\volmov(X)=-(\PP_-)^n=0$. By Proposition \ref{self_inter}, we
have that $\PP_-=0$. Let $f:Y\to X$ be a movable modification,
$g:Z\to Y$ be any $\QQ$-factorial model over $Y$ and $\phi=f\circ
g$. Then $0=(\PP_-)_Y=A_{Y/X}$, hence $K_Y+E_f=f^*K_X$. In
particular, $f^*K_X$ is $\QQ$-Cartier. For the trace on $Z$, we have
$$\begin{aligned}&0=(\PP_-)_Z=P_\sigma(g^*(K_Y+E_f))-\phi^*K_X\\=\; & P_\sigma(g^*(f^*K_X))+P_\sigma(-g^*(f^*K_X))\\
\leq\; & g^*(f^*K_X)-g^*(f^*K_X)=0,\end{aligned}$$ where the second
row follows from Lemma \ref{double_pullback}. In particular, we
obtain that $\Env_X(-K_X)=\PP_\sigma(-f^*K_X)$ is a $\QQ$-Cartier
$b$-divisor determined by $f$. Thus, $-f^*K_X$ is $f$-nef by
Lemma \ref{nef_determinant}. Similarly, since $P_\sigma(g^*(f^*K_X))=g^*(f^*K_X)$, we have that $f^*K_X$ is also $f$-nef. Thus, $f^*K_X$ is $f$-numerically trivial and $X$ is numerically $\QQ$-Gorenstein. The inequality $0=\PP_-\leq A_{\XX/X}$ yields that $X$ is numerically log canonical. We may apply Proposition \ref{num_lc} to conclude the proof. The other direction is obvious.
\end{proof}

\section{Finite Endomorphisms}

We briefly review the pullback of Weil $b$-divisors.

Let $\phi:Y\to X$ be a generically finite dominant morphism between normal varieties.  Every divisorial valuation $\nu$ on $Y$ induces a divisorial valuation $\phi_*\nu$ via the natural inclusion of the function field $\phi^*:k(X)\hookrightarrow k(Y)$ given by $\phi_*\nu(f)=\nu(f\circ \phi)$ \cite[Lemma 1.13]{BdFF}. In other words, suppose that $F$ is a prime divisor on $X'$ over $X$. Then there is a birational model $Y'$ over $Y$ such that $\phi$ lifts to a morphism $\phi':Y'\to X'$. We may obtain such $Y'$ by blowing up the indeterminancy of the rational map $Y\dashrightarrow X'$. If $E$ is an irreducible component of $(\phi')^{-1}F$ such that $\phi'(E)=F$, then $\phi_*\nu_E$ is a scalar multiple of $\nu_F$ with the scalar $\nu_E(\phi'^*F)$.

Let $W$ be an $\RR$-Weil $b$-divisor on $X$. We define the pullback of $W$ as the $\RR$-Weil $b$-divisor $\phi^*W$ such that $\nu_E(\phi^*W)=(\phi_*\nu_E)(W)$.

For a normal variety $Z$, we denote by $K_\ZZ$ the Weil $b$-divisor given by a canonical divisor on each model and $E_\ZZ$ the Weil $b$-divisor given by the reduced exceptional divisor over $Z$ on each model.

\begin{lem}\label{finite_morp} \cite[Lemma 2.19 and 3.3]{BdFF}
Let $\phi:Y\to X$ be a finite dominant morphism between normal varieties. Then
\begin{enumerate}
\item If $D$ is an $\RR$-Weil divisor on $X$, then $\Env_Y(\phi^*D)=\phi^*\Env_X(D)$.
\item If $\WW$ is an $\RR$-Weil $b$-divisor over $X$ such that
$\Env_\XX(\WW)$ is well-defined, then
$\Env_\YY(\phi^*\WW)=\phi^*\Env_\XX(\WW)$.
\item If $E$ is a prime exceptional divisor over $Y$, then $$\nu_E(K_\YY+E_\YY)=\nu_E(\phi^*(K_\XX+E_\XX)).$$
\item If $\phi$ is \'etale in codimension 1, then $K_\YY+E_\YY=\phi^*(K_\XX+E_\XX)$.
\end{enumerate}
\end{lem}

\begin{proof}
Part (1) and (2) are the same as \cite[Lemma 2.19]{BdFF}.

Let $E$ be an exceptional divisor over $Y$. Let $X'$ be a smooth model over $X$  such that the center of $\phi_*\nu_E$ on $X'$ is a divisor $F$. Let $Y'$ be a smooth model over $Y$ such that $E$ is a divisor on $Y'$ and $\phi':Y'\dashrightarrow X'$ is a morphism. We summarize this in the following diagram:
$$\xymatrix{
**[l]E\subset Y'\ar^{\phi'}[r]\ar_{g}[d]&**[r]X'\supset F\ar^{f}[d]\\
Y\ar^{\phi}[r]&**[r]X. }$$ We have $\phi_*\nu_E=b\nu_F$ where
$b=\nu_E(\phi'^*F)$. Hence the ramification order of $\phi'$ at the
generic point of $E$ is $b-1$. We get
$$\nu_E(K_{Y'})=\nu_E(\phi'^*K_{X'})+b-1=b\nu_F(K_{X'})+b-1.$$ Thus
$$\nu_E(K_\YY+E_\YY)=\nu_E(K_{Y'})+1=b\nu_F(K_{X'})+b=\phi_*\nu_E(K_\XX+E_\XX).$$
Part (3) follows.

If $\phi$ is \'etale in codimension 1, then $K_Y=\phi^*K_X$. Hence,
$K_\YY+E_\YY$ and $\phi^*(K_\XX+E_\XX)$ also coincide on
non-exceptional valuations.
\end{proof}

\begin{lem}\label{env_vs_sigma}
Let $f:Y\to X$ be a movable modification. Then
$$\Env_\XX(K_\XX+E_\XX)=\PP_\sigma(K_Y+E_f).$$
In particular, $\PP_-=\Env_\XX(K_\XX+E_\XX)+\Env_X(-K_X)$.
\end{lem}

\begin{proof}
The proof is similar to Lemma \ref{trace_on_mov}. Let $g:Z\to X$ be
any $\QQ$-factorial model factoring through $f$ via $\phi:Z\to Y$.
By Lemma \ref{well_def}, the trace on $Z$ satisfies:
$$(\PP_\sigma(K_Y+E_f))_Z=\PP_\sigma(K_Z+E_g)\leq K_Z+E_g.$$
Hence, $(\Env_\XX(K_\XX+E_\XX))_Z\leq(\PP_\sigma(K_Y+E_f))_Z.$ On
the other hand, $\PP_\sigma(K_Y+E_f)\leq K_\XX+E_\XX$. By the
definition of nef envelope, we have $\Env_\XX(K_\XX+E_\XX)\geq
\PP_\sigma(K_Y+E_f)$.
\end{proof}


We are ready to prove the main theorem.

\begin{thm}\label{main_finite}
Let $\phi:(Y,0)\to (X,0)$ be a finite morphism of degree $d$ between
normal isolated singularities such that $\phi$ is \'etale in codimension 1. Then
\begin{enumerate}
\item $\PP_{-,Y}= \phi^*\PP_{-,X}$.
\item $\volmov(Y)= d\volmov(X)$.
\end{enumerate}
\end{thm}
\begin{proof}
By Lemma \ref{finite_morp} and \ref{env_vs_sigma},
$$\begin{aligned}
\PP_{-,Y}&=\Env_\YY(K_\YY+E_\YY)+\Env_Y(-K_Y) \\
&= \Env_\YY(\phi^*(K_\XX+E_\XX))+\Env_Y(-\phi^*K_X) \\
&= \phi^*\Env_\XX(K_\XX+E_\XX))+\phi^*\Env_X(-K_X) \\
&= \phi^*\PP_{-,X}.
\end{aligned}
$$

Part (2) now follows from Lemma \ref{self_inter}.
%
\end{proof}

\begin{thm}\label{main_volmov=0}
If $\phi:(X,0)\to (X,0)$ is a finite endomorphism of degree $d \geq
2$ such that $\phi$ is \'etale in codimension 1, then $\volmov(X)=0$.
\end{thm}
\begin{proof}
By Theorem \ref{main_finite}, $\volmov(X)= d\volmov(X)$. But
$\volmov(X)$ is a nonnegative finite number. There must be
$\volmov(X)=0$.
\end{proof}

\end{document}